\documentclass[leqno,draft]{amsart}
\def\today{May 22, 2012}
\date{\today}
\usepackage[usenames]{color}
\newcommand{\pmt}[1]{{\begin{pmatrix} #1  \end{pmatrix}}}

\renewcommand{\L}{\mathcal{L}}

\newcommand{\D}{\mathbb{D}}

\newcommand{\C}{\mathbb{C}}
\newcommand{\R}{\mathbb{R}}

\newcommand{\SL}{\operatorname{SL}}

\newcommand{\SU}{\operatorname{SU}}
\newcommand{\PSL}{\operatorname{PSL}}

\renewcommand{\Re}{\operatorname{Re}}

\newcommand{\trans}[1]{\vphantom{#1}^t#1}

\numberwithin{equation}{section}
\newtheorem{theorem}{Theorem}[section]

\newtheorem{lemma}[theorem]{Lemma}

\newtheorem{proposition}[theorem]{Proposition}
\newtheorem*{introtheorem}{Theorem}
\newtheorem*{mainlemma*}{Main Lemma}
\newtheorem*{keylemma*}{Key Lemma}

\theoremstyle{definition}
\newtheorem{remark}[theorem]{Remark}
\newtheorem*{remark*}{Remark}

\newtheorem*{question}{Question}

\title[
Flat surfaces whose hyperbolic Gauss maps
are bounded]{%
Flat surfaces  
in hyperbolic 3-space
whose hyperbolic Gauss maps
are bounded}%
\author[F. Martin]{Francisco Mart\'\i{}n}
\address[Martin]{
  Departamento de Geometr\'\i{}a y Topolog\'\i{}a,
  Universidad de Granada,
  18071 Granada, Spain.
}
\email{fmartin@ugr.es}
\author[M. Umehara]{Masaaki Umehara}
\address[Umehara]{%
   Department of Mathematical and Computing Sciences,
   Tokyo Institute of Technology
   2-12-1-W8-34, O-okayama, Meguro-ku,
   Tokyo 152-8552, Japan.
}
\email{umehara@is.titech.ac.jp}
\author[K. Yamada]{Kotaro Yamada}
\address[Yamada]{%
   Department of Mathematics,
   Tokyo Institute of Technology
   2-12-1-H-7, O-okayama, Meguro-ku,
   Tokyo 152-8551, Japan.
}
\email{kotaro@math.titech.ac.jp}
\thanks{
  The first author is partially
  supported by MEC-FEDER Grants No. MTM2007 - 61775 and No. MTM2011-22547,  and a Regional
  J. Andaluc\'\i a Grant no. P09-FQM-5088.
  The second and the third authors  are 
  partially supported by Grant-in-Aid for 
  Scientific Research (A) No.~22244006, 
  and Scientific Research (B) No.~21340016,
  respectively, from the Japan Society for the Promotion of Science.
}

\subjclass[2010]{%
  Primary 53D10; 
  Secondary 53A10, 
	    53A35. 
}

\begin{document}
\maketitle
\begin{abstract}
We  construct a weakly complete flat surface
in hyperbolic 3-space $H^3$ having 
a pair of hyperbolic Gauss maps
both of whose images are
contained in an arbitrarily given
open disc in the ideal boundary
of $H^3$. 
This construction 
is accomplished 
as an application of the minimal surface theory.
This looks an interesting phenomenon
if one comparing the
fact that there are no complete minimal 
(resp.\ constant mean curvature one)
surfaces in $\R^3$ (resp. $H^3$)
having bounded Gauss maps (resp.\ bounded hyperbolic Gauss maps). 
\end{abstract}

\section{Introduction}
It is a classical fact that any complete immersed flat surface 
in the hyperbolic 
3-space $H^3$ must be a horosphere or a hyperbolic cylinder,
where \lq flat\rq\ means that the Gaussian curvature vanishes identically. 
However, this does not imply the lack of an interesting 
global theory for flat surfaces.
G\'alvez, Mart\'inez and Mil\'an \cite{GMM} established a Weierstrass-type
representation formula for such surfaces.
In this paper, we discuss on flat surfaces 
with admissible singularities.
(A singular (i.e., degenerate) point is called {\it admissible} 
if the corresponding
points on nearby parallel surfaces are regularly immersed. See \cite{KUY}.)
Flat surfaces with admissible singularities 
in $H^3=\SL(2,\C)/\SU(2)$
are called {\it flat fronts}
which can be characterized as the projections of holomorphic
immersed Legendrian curves in $\SL(2,\C)$. Here,
a holomorphic map $\L\colon{}\D_1\to\SL(2,\C)$ is called 
{\em Legendrian\/}
if the pull-back $\L^*\Omega_{\SL}$ vanishes on $\D_1$, 
where $\Omega_{\SL}$ is the complex contact form on $\SL(2,\C)$
defined as 
\begin{equation}\label{eq:contact-sl}
    \Omega_{\SL}:=x_{11}dx_{22}-x_{21}dx_{12}
\qquad ((x_{ij})_{i,j=1,2}\in \SL(2,\C)).
\end{equation}
A flat front in $H^3$ induces a pair of
hyperbolic Gauss maps $(G_+,G_-)$, both of which are
holomorphic mappings into $\C\cup \{\infty\}$
as follows: 
Let $f:M^2\to H^3$ be a flat front
defined on a Riemann surface $M^2$.
The normal geodesic passing through $f(p)$
meets the ideal boundary $\partial H^3$
of the hyperbolic 3-space $H^3$
at $G_\pm (p)$
if one identify $\partial H^3$ by $\C\cup \{\infty\}$
via the Poincar\'e half-space model of the hyperbolic 3-space.
In \cite{KUY}, two hyperbolic Gauss maps $G_+$
and $G_-$ are indicated by $G$ and $G_*$,
respectively.
We are interested in the behavior of hyperbolic Gauss maps
of flat surfaces in $H^3$.
For flat fronts in $H^3$, the completeness
and the weak completeness are defined (cf. \cite[\S2]{KRUY2}). 
Completeness implies weak
completeness.
There are many complete or weakly complete flat fronts in $H^3$
as shown in \cite{KUY,KRUY1,KRUY2}.

It is well-known that the Gauss map 
(resp.\ hyperbolic Gauss map)
of a complete immersed minimal surface in $\R^3$
(resp.\ 
a complete immersed constant mean curvature one surface in $H^3$)
can omit at most four points (cf. Fujimoto \cite{F1}, \cite{F2} and 
Yu \cite{Y}).
It is well-known that
conformal minimal immersions are
obtained by taking the real part of
null holomorphic immersions.
Recently, as an improvement of 
Nadirashvili's discovery of
complete bounded minimal surfaces in $\R^3$,
the existence of
complete bounded null holomorphic immersion
\begin{equation}\label{eq:1}
   F:\D_1\longrightarrow \C^3
\end{equation}
of the unit disk $\D_1\subset\C$
is shown (cf. \cite{AL}), 
where {\it null\/} means that $F_z\cdot F_z$ vanishes
identically, here $F_z:=dF/dz$ is the derivative of 
$F$ with respect to the complex coordinate $z$ of $\D_1$
and the dot denotes the canonical complex bilinear form.
In fact, properly immersed 
null holomorphic curves in $\C^3$
of arbitrary topology are constructed in 
Alarc\'on and L\'opez \cite{AL}.

It is known that null curves in $\C^3$ are
closely related to Legendrian curves in $\C^3$ 
(cf.\ Bryant \cite{B} and also Ejiri-Takahashi \cite{ET}
for the corresponding $\SL(2,\C)$-case).
It can be easily checked
a holomorphic immersion
$\L:\D_1\to \SL(2,\C)$ is 
Legendrian if 
$\L^{-1}d\L$ is off-diagonal, namely,
there exist two
holomorphic 1-forms $\omega$ and $\theta$
on $\D_1$
such that
\begin{equation}\label{eq:L-intro}
\L^{-1}d\L=\pmt{0 & \theta \\ \omega & 0}.
\end{equation}
As pointed out in the above paragraph,
the projection of $\L$ into the hyperbolic
$3$-space gives a flat front in 
$H^3=\SL(2,\C)/\SU(2)$.
Then the singular set of this flat surface
in $H^3$ is given by
$$
\{z\in \D_1\,;\, |\rho(z)|=1\},
$$
where $\rho$ is the meromorphic
function defined by $\rho:=\theta/\omega$
called the {\it ratio of canonical forms}
 (cf.\ \cite{KUY,KRUY1}).
By Darboux's theorem, the contact structure of $\SL(2,\C)$
is locally equivalent to that of $\C^3$.
Moreover, the following explicit transformation
\begin{equation}\label{eq:T}
    \mathcal T\colon{}\C^3\ni (X,Y,Z)\longmapsto
            \begin{pmatrix}
	     e^{-Z} & Y e^Z\\
	     X e^{-Z} & (1+XY)e^Z
	    \end{pmatrix}\in\SL(2,\C)
\end{equation}
maps holomorphic Legendrian curves in 
$\C^3$
with contact form 
\begin{equation}\label{eq:contact-c}
\Omega_{C}:=dZ+Y dX
\end{equation}
to those in $\SL(2,\C)$.
Using this transformation,
we prove the following assertion.

\begin{introtheorem}
 There exists a weakly complete flat front
 in hyperbolic $3$-space whose induced pair of
 hyperbolic Gauss maps are contained in
 an arbitrarily given
 open disk in the ideal boundary
 of $H^3$. 
\end{introtheorem}
It should be remarked that there are no
compact flat fronts in $H^3$
(cf.\ \cite[Proposition 3.6]{KUY}).
Also, the assumption of
weak completeness in the theorem
is crucial, since 
two hyperbolic Gauss maps
can omit at most finite points
if the given flat front
is complete (see Remark \ref{rmk:complete}).
In contrast to this theorem,
Kawakami \cite{Kaw} showed the 
ratio of canonical forms $\rho$ 
of weakly complete flat fronts can omit at most three 
exceptional values.

\section{Proof of the theorem}
\label{sec:1}

This section is devoted to prove the  theorem
as we stated in the introduction.

Recall that
\begin{equation}\label{eq:Q}
   H^3:=\SL(2,\C)/\SU(2)\\
=\{a a^*\,;\,a\in\SL(2,\C)\}\qquad (a^*=\trans{\bar a}).
\end{equation}
is the hyperbolic $3$-space of constant curvature $-1$.
A smooth map $f\colon{}\D_1 \to H^3$ 
is called a ({\em wave}) {\em front\/}
if there exists a Legendrian immersion
$L_f\colon{}\D_1\to T^*_1H^3$ with respect 
to the canonical contact structure of the unit
cotangent bundle $\pi:T^*_1H^3\to H^3$ such 
that $\pi\circ L_f=f$.
For a holomorphic Legendrian immersion $\L\colon{}\D_1\to\SL(2,\C)$,
the projection
\begin{equation}\label{eq:flat-front}
   f:=\L\L^*\colon{}\D_1\longrightarrow H^3
\end{equation}
gives a {\em flat front\/} in $H^3$
 (see \cite{KUY,KRUY1}  for the definition of
flat fronts).
In particular, the Gaussian curvature vanishes at
each point where $f$ is an immersion.
We call $\L$ in \eqref{eq:flat-front} the {\em holomorphic lift\/}
of $f$.
A flat front $f$ is called {\em weakly complete\/}
if its holomorphic lift is complete 
with respect to the 
pull-back metric 
\begin{equation}\label{eq:ds^2_L}
ds^2_{\L}= |\omega|^2+|\theta|^2\qquad 
(|\omega|^2:=\omega\bar \omega,\,\,\,|\theta|^2:=\theta\bar \theta)
\end{equation}
of
the canonical Hermitian metric 
of $\SL(2,\C)$ by $\L$,
where 
$\omega$ and $\theta$ are
holomorphic 1-forms 
satisfying  \eqref{eq:L-intro}
defined on $\D_1$
(cf.\ \cite{KUY,KRUY1}).
The first fundamental form of $f$
is written as
\begin{equation}\label{eq:first}
    ds^2_{f}:= |\omega|^2+|\theta|^2 +\omega\theta+\bar\omega\bar\theta 
           = |\omega+\bar\theta|^2.
\end{equation}
On the other hand, the pair of hyperbolic Gauss 
maps of $f$ is given by
\begin{equation}\label{eq:Gpm}
G_+:=\frac{A}{C},\qquad G_-:=\frac{B}{D},
\end{equation}
where 
$$
\L:=\pmt{A & B \\ C & D}.
$$

\medskip
We now prove the theorem in the introduction:
Let 
$$
F=(X,Y,Z):\D_1\longrightarrow \C^3
$$
be a bounded null holomorphic immersion
whose induced metric is complete.
Without loss of generality, we may
assume that
\begin{equation}\label{eq:XY}
1<|X|<2,\qquad |Y|<\frac13
\end{equation}
hold on $\D_1$.
Let 
$(g,\eta\,dz)$  be the Weierstrass data 
of $F$, that is,
$$
F_z
\left(:=\frac{d F}{dz}\right)
=\frac12(1-g^2,\sqrt{-1}(1+g^2),2g)\eta.
$$
The metric induced by $F$ is written as
$$
\frac12(1+|g|^2)^2|\eta dz|^2.
$$
By \eqref{eq:XY},
the projection $\hat F:=(X,Y)\colon{}\D_1\to \C^2$
of $F$  is a bounded holomorphic map.
Moreover, the following assertion holds:

\begin{lemma}\label{lem:dsigma}
The metric 
\begin{equation}\label{eq:dsigma}
d\sigma^2:=|dX|^2+|dY|^2,
\end{equation}
induced by $\hat F$ is 
a complete Riemannian metric on $\D_1$.
In particular, $\hat F$ is a holomorphic
immersion.
\end{lemma}

\begin{proof}
We have that
\begin{align*}
2d\sigma^2&
=2\left|\hat F_z\right|^2|dz|^2
      =\frac12\bigl(|1-g^2|^2+|1+g^2|^2\bigr)|\eta dz|^2
      =(1+|g|^4)|\eta dz|^2\\
&\ge \frac12(1+|g|^2)^2|\eta dz|^2
      =\left|F_z\right|^2|dz|^2
=|dX|^2+|dY|^2+|dZ|^2.
\end{align*}
Since $F$ is a complete immersion,
$d\sigma^2$ is a complete Riemannian metric.
\end{proof}

We now consider a
new holomorphic immersion defined by
$$
\tilde F:=(X,Y,W):\D_1\longrightarrow \C^3
$$
where
$$
W:=-\int_0^z Y dX.
$$
Then $\tilde F$
gives a holomorphic Legendrian immersion
with respect to the contact form $\Omega_C$ 
as in \eqref{eq:contact-c}.
Then the induced map (see \eqref{eq:T} for the definition of 
$\mathcal T$) 
$$
\L:=\mathcal T\circ \tilde F:\D_1\longrightarrow \SL(2,\C)
$$
can be written by
\begin{equation}\label{eq:L}
\L=\pmt{e^{-W} & Y e^W \\ X e^{-W} & (1+XY)e^W},
\end{equation}
and the mapping
$f:\D_1\to H^3$ given by 
\eqref{eq:flat-front} is a
flat front.
In fact, by a straightforward calculation,
we have that
\begin{equation}\label{eq:dL}
\L^{-1}d\L=
\pmt{
-(YdX+dW) & e^{2W}(dY-Y^2dX)\\
e^{-2W}dX & YdX+dW
}.
\end{equation}
Since $YdX+dW=\tilde F^*\Omega_C$ vanishes
identically, the $sl(2,\C)$-valued
$1$-form $\L^{-1}d\L$ is off-diagonal
(i.e. we just checked that $\L$ is Legendrian).
The following assertion holds:

\begin{proposition}\label{prop:bdd}
The images of two hyperbolic Gauss maps
$G_+$ and $G_-$ associated to $f$
lie in the unit disk $\{\xi\in \C\,;\, |\xi|<1\}$.
\end{proposition}

\begin{proof}
By \eqref{eq:Gpm} and \eqref{eq:L}, 
we have that
$$
G_+=\frac1{X},\qquad G_-=\frac{Y}{1+XY}.
$$
Then the inequalities \eqref{eq:XY}
yield the assertion.
\end{proof}

To prove the completeness of the metric
$ds^2_{\L}$, we prepare the following 
assertion:

\begin{lemma}\label{lem:key}
The metric $ds^2_{\L}$ is positive definite
and satisfies the inequality
\begin{equation}\label{eq:kuy}
ds^2_f\le 2 ds^2_{\L}.
\end{equation}
\end{lemma}

\begin{proof}
Since $\mathcal T$ is a local diffeomorphism,
$\L$ is an immersion, and 
$ds^2_{\L}$ is positive definite.
The inequality \eqref{eq:kuy}
is obtained as 
follows (cf. \eqref{eq:ds^2_L} and \eqref{eq:first})
$$
ds^2_{f}= |\omega|^2+|\theta|^2 +\omega\theta+\bar\omega\bar\theta 
\le 2|\omega|\,|\theta|+|\omega|^2+|\theta|^2
\le 2(|\omega|^2+|\theta|^2)=2ds^2_{\L}.
$$ 
\end{proof}

The following
assertion is a key to prove the
theorem.

\begin{proposition}\label{prop:key}
The metric $ds^2_{\L}$ is complete.
\end{proposition}

\begin{proof}
We fix a piecewise smooth
divergent path
$
\gamma:[0,\infty)\to\D_1
$
arbitrarily.
It is sufficient to show that
the image of $\gamma$ has
infinite length with respect to
$ds^2_{\L}$.
If $f(\gamma([0,\infty)))$ is unbounded
in $H^3$,
then the path $\gamma$ must have
infinite length with respect to $ds^2_f$
because of the completeness of
the hyperbolic space $H^3$.
Then the inequality \eqref{eq:kuy}
implies that $\gamma$ has
infinite length with respect to
$ds^2_{\L}$.
So we assume that 
$f(\gamma([0,\infty)))$ is bounded
in $H^3$.
Since $\SU(2)$ is compact,
\eqref{eq:Q} yields that
the image 
$\L(\gamma([0,\infty)))$ is bounded
in $\SL(2,\C)$.
By \eqref{eq:L}, there exists
a positive constant $m$ such that
$$
|e^{-W}|<m,\qquad \left|(1+XY)e^{W}\right|<\frac{m}3
$$
holds on $\gamma$. 
Using \eqref{eq:XY}, we have that
\begin{equation}\label{eq:m}
\frac1{m}\le |e^{W(\gamma(t))}|\le m
\qquad (t\ge 0).
\end{equation}
On the other hand, it holds 
that (cf. \eqref{eq:dL})
$$
\L^{-1}d\L=
\pmt{
0 & e^{2W}(dY-Y^2dX) \\
e^{-2W}dX & 0
}.
$$
By \eqref{eq:ds^2_L},
it holds
along $\gamma$ that
\begin{align*}
ds^2_{\L}
&=
|e^{-4W}|\,|dX|^2+
|e^{4W}|\,|dY-Y^2dX|^2 \\
&\ge
\frac1{m^4}\biggl(|dX|^2+
|dY-Y^2dX|^2\biggr) \\
&\ge
\frac1{m^4}\biggl(|dX|^2+
|dY|^2-2(2|Y^2dX|)\frac{|dY|}2+|Y|^4\,|dX|^2\biggr) \\
&\ge
\frac1{m^4}\biggl(|dX|^2+
|dY|^2-(4|Y^2dX|^2+\frac{|dY|^2}4 )+|Y|^4\,|dX|^2\biggr) \\
&\ge
\frac1{m^4}
\biggl((1-3|Y|^4)|dX|^2+\frac{3|dY|^2}4\biggr). 
\end{align*}
Again using \eqref{eq:XY}, 
$$
ds^2_{\L}
\ge \frac1{m^4}
\biggl(\frac{26|dX|^2}{27}+\frac{3|dY|^2}4\biggr) 
\ge \frac3{4m^4}
\biggl(|dX|^2+|dY|^2\biggr)=\frac3{4m^4}d\sigma^2 
$$
hold on $\gamma$.
Since $d\sigma^2$ is a complete metric 
(cf. Lemma \ref{lem:dsigma}),
we get the conclusion. 
\end{proof}

\begin{proof}[Proof of the theorem]
Let $r$ be an arbitrarily sufficiently
small positive number.
We set
$$
\L_r:=\pmt{r & 0 \\ 0 & r^{-1}}\L.
$$
Then $f_r:=\L_r\L_r^*$ is a
desired flat front.
In fact, $f_r$ is weakly complete
since the induced metric of $\L_r$
is complete (cf. Proposition \ref{prop:key}).
On the other hand,
the hyperbolic Gauss maps of 
$f_r$ are equal to $r^2G_+$ and $r^2G_-$,
and their images are contained in a
disk of radius $r^2$
(cf. Proposition \ref{prop:bdd}).
\end{proof}

\begin{remark}
\label{rmk:complete}
Let $(G_+,G_-)$ be the pair of
hyperbolic Gauss maps
induced by a complete flat front
$f:M^2\to H^3$, where
$M^2$ is a Riemann surface.
By \cite[Lemma 3.3]{KUY},
there exists a closed Riemann surface
$\bar M^2$ and finite points
$p_1,...,p_n$ on $\bar M^2$
such that
$M^2$ is bi-holomorphic to
$\bar M^2\setminus \{p_1,...,p_n\}$.
If $G_+$ (resp.\ $G_-$) has an essential singularity
at some $p_j$ (i.e. $p_j$ is an irregular end),
Picard's theorem implies that $G_+$ (resp.\ $G_-$)
omits at most two points in $\C\cup\{\infty\}$.
Otherwise, both $G_+$ and $G_-$ are meromorphic 
functions on a compact Riemann surface $\bar M^2$.
Hence both $G_+$ and $G_-$ are surjective maps
onto $\C\cup\{\infty\}$.
\end{remark}

We now remark that
the following problem seems interesting
as an analogue of Calabi-Yau problem
in minimal surface theory:

\begin{question} 
Are there complete bounded 
holomorphic Legendrian curves
immersed in $\C^3$?
\end{question}

As seen in our previous arguments,
this problem is related to our
main result.
Moreover, it is also 
closely related
to the existence of bounded
weakly complete improper affine fronts
in the affine space $\R^3$:
A notion of {\em IA-maps\/}
in the affine $3$-space has been introduced 
by A. Mart\'{\i}nez  \cite{Martinez}.
IA-maps  are improper affine spheres
with a certain kind of singularities.
Since all of IA-maps
are wave fronts (see \cite{Nakajo, UY}),
we call them  {\em improper affine fronts}
(the terminology \lq improper affine fronts\rq\
 has been already used in
Kawakami-Nakajo \cite{KN}).
The precise definition of improper affine fronts
is given in \cite[Remark 4.3]{UY}.
In \cite{UY}, {\it weak completeness\/} of improper affine fronts
is introduced.
Then we have
\begin{proposition}\label{thm:affine}
The existence of a complete bounded immersed 
Legendrian curve as in the above question
would imply
the existence of a weakly complete improper affine front
whose image is bounded.
\end{proposition}

\begin{proof}
 Let $F=(X,Y,Z)$ be a complete bounded Legendrian immersion into $\C^3$.
 Since $F$ is Legendrian, \eqref{eq:contact-c} yields that
 $dZ = -Y\,dX$.
 Here, by completeness of $F$, the induced metric
 \[
   ds^2_F=
     |dX|^2+|dY|^2+|dZ|^2 
   = |dX|^2+|dY|^2+|YdX|^2 
   =  (|Y|^2+1)|dX|^2+|dY|^2
 \]
 is complete.
 Moreover, since the image of $F$ is bounded, we have
 \[
     ds^2_F\leq C(|dX|^2+|dY|^2)\qquad
     (\text{$C> 0$ is a constant}).
 \]
 Thus, the metric
 \begin{equation}\label{eq:d-tau}
     d\tau^2 := |dX|^2 + |dY|^2
 \end{equation}
 is complete.
Hence, we have the following improper affine front $f$
substituting the pair of holomorphic functions
$(X,Y)$ into Martinez' representation 
formula \cite[Theorem 3]{Martinez}
 \begin{align}
    f&= \left(X+\overline Y,
            \frac{1}{2}(|X|^2-|Y|^2)+
    \Re\left(XY-2\int Y\,dX\right)\right)\label{eq:ia-map}\\
     &= \left(X+\overline Y,
            \frac{1}{2}(|X|^2-|Y|^2)+
            \Re\left(XY+2Z\right)\right)\colon{}\D_1\to\R^3.
  \nonumber
 \end{align}
 Since $d\tau^2$ in \eqref{eq:d-tau} is complete,
 $f$ is weakly complete,
 by definition of weak completeness given in \cite{UY}.
 The boundedness of $f$ follows from that of $F$.
\end{proof}

\end{document}